\documentclass[12pt,reqno]{amsart}
\usepackage{graphicx}
\usepackage{amssymb}

\numberwithin{equation}{section}
\newtheorem{thm}{Theorem}[section]
\newtheorem{cor}[thm]{Corollary}
\newtheorem{lem}[thm]{Lemma}
\newtheorem{prop}[thm]{Proposition}

\theoremstyle{definition}
\newtheorem{defn}[thm]{Definition}
\theoremstyle{remark}
\newtheorem{rem}[thm]{Remark}
\newtheorem{notation}[thm]{Notation}
\numberwithin{equation}{section}

\newcommand\Hom{\operatorname{Hom}}

\newcommand\Ext{\operatorname{Ext}}

\newcommand\ara{\operatorname{ara}}
\newcommand\height{\operatorname{height}}
\newcommand\cd{\operatorname{cd}}
\newcommand\Rad{\operatorname{Rad}}
\newcommand\Ann{\operatorname{Ann}}
\newcommand\Supp{\operatorname{Supp}}
\newcommand\Spec{\operatorname{Spec}}
\newcommand\Tor{\operatorname{Tor}}

\newcommand{\qism}{\stackrel{\sim}{\longrightarrow}}
\newcommand\grade{\operatorname{grade}}
\newcommand\depth{\operatorname{depth}}
\newcommand\id{\operatorname{id}}

\begin{document}

\title[Cohomologically Complete Intersections]{On Cohomologically Complete Intersections in Cohen-Macaulay Rings}%
\author{Waqas Mahmood}%
\address{Abdus Salam School of Mathematical Sciences, GCU, Lahore Pakistan}%
\email{waqassms$@$gmail.com}%

\thanks{This research was partially supported by Higher Education Commission, Pakistan}
\begin{abstract}
An ideal $I$ of a local Cohen-Macaulay ring $(R,{\mathfrak m})$  is called a cohomologically complete intersection if $H^i_I(R)= 0$ for all $i\neq c := \height(I)$. Here $H^i_I(R)$, $i\in\mathbb{Z}$ denotes the local cohomology of $R$ with respect to $I$. For instance, a set-theoretic complete intersection is a cohomologically complete intersection. Here we study cohomologically complete intersections from various homological points of view. As a main result it is shown that the vanishing $H^i_I(M) = 0$ for all $i \not= c$ is completely encoded in homological properties of $H^c_I(M)$.  These results extend those of Hellus and Schenzel (see \cite[Theorem 0.1]{pet1}) shown in the case of a local Gorenstein ring. In particular we get a characterization of cohomologically complete intersections in a Cohen-Macaulay ring
in terms of the canonical module.
\end{abstract}
\maketitle
\section{Introduction}
Let $(R, \mathfrak m)$ denote a local Noetherian ring. For an ideal
$I \subset R$ it is a rather difficult question to determine the smallest
number $n\in \mathbb N$ of elements $a_1,..., a_n \in R$ such that $\Rad I =
\Rad (a_1,...,a_n)R.$ This number is called the arithmetic rank, $\ara I,$
of $I.$ By Krull's generalized principal ideal theorem it follows that
$\ara I\geq \height(I).$ Of a particular interest is the case whenever $\ara I = \height(I).$
If this equality holds then $I$ is called a set-theoretic complete intersection.

For the ideal $I$ let $H^i_I(\cdot), i \in \mathbb Z,$ denote the local cohomology
functor with respect to $I,$ see \cite{goth} for its definition and basic results. The
cohomological dimension, $\cd(I),$ defined by
\[
 \cd(I)= \sup     \{ i \in \mathbb Z | H^i_I(R) \not= 0\}
\]
is another invariant related to the ideal $I.$ It is well known that
\[
\grade I \leq \height(I) \leq \cd(I)\leq \ara I.
\]
In particular, if $I$ is a set-theoretically complete intersection
it follows that $\height(I) = \cd(I)$. The converse does not hold in general (see   \cite[Remark 2.1(ii)]{he1}). Not
so much is known about ideals with the property of $ \grade I = \cd(I)$. We call those
ideals cohomologically complete intersections. In this paper we continue with
the investigations of cohomologically complete intersections, in particular when
$I$ is an ideal in a Cohen-Macaulay ring $(R, \mathfrak m).$ In this case $I$ is a cohomologically complete intersection if $\height(I) = \cd(I)$.

As an application of our main results there is a characterization of
cohomologically complete intersections. In fact,
this provides a large number of necessary conditions for an ideal to be a
set-theoretic complete intersection in a local Cohen-Macaulay ring.

\begin{thm} \label{1.1}
Let $(R,\mathfrak m)$ be a ring of dimension $n, I\subseteq R$ be an ideal of
$\grade(I,M)= c$, and $M\neq 0$ be a maximal Cohen-Macaulay module with $\id_R(M)< \infty$. Then the following
conditions are equivalent:
\begin{itemize}
\item[(a)] $H^i_I(M)= 0$ for all $i\neq c.$
\item[(b)] For all ${\mathfrak p}\in V(I)\cap \Supp_R(M)$ the natural homomorphism
\[
H^{h({\mathfrak p})}_{{\mathfrak p}R_{\mathfrak p}}(H^c_{IR_{\mathfrak p}}(M_{\mathfrak p})) \to H^{\dim(M_{\mathfrak p})}_{{\mathfrak p}R_{\mathfrak p}}(M_{\mathfrak p})
\]
is an isomorphism and $H^i_{{\mathfrak p}R_{\mathfrak p}}(H^c_{IR_{\mathfrak p}}(M_{\mathfrak p}))= 0$ for all $i\neq h({\mathfrak p})=\dim(M_{\mathfrak p})- c .$
\item[(c)] For all ${\mathfrak p}\in V(I)\cap \Supp_R(M)$  the natural homomorphism
\[
\Ext^{h({\mathfrak p})}_{R_{\mathfrak p}}(k({\mathfrak p}),H^c_{IR_{\mathfrak p}}(M_{\mathfrak p}))\rightarrow \Ext^{\dim(M_{\mathfrak p})}_{R_{\mathfrak p}}(k({\mathfrak p}), M_{\mathfrak p})
\]
is an isomorphism and $\Ext^i_{R_{\mathfrak p}}(k({\mathfrak p}), H^c_{IR_{\mathfrak p}}(M_{\mathfrak p}))= 0$ for all $i\neq h({\mathfrak p})=\dim(M_{\mathfrak p})- c$.
\item[(d)] For all ${\mathfrak p}\in V(I)\cap \Supp_R(M)$ the natural homomorphism
\[
K({\hat{M}}_{\mathfrak p})\rightarrow \Ext^{c}_{{\hat{R}}_{\mathfrak p}}(H^c_{I{{\hat{R}}_{\mathfrak p}}}({\hat{M}}_{\mathfrak p}), K({\hat{R}}_{\mathfrak p}))
\]
is an isomorphism and $\Ext^{i}_{{\hat{R}}_{\mathfrak p}}(H^c_{I{{\hat{R}}_{\mathfrak p}}}({\hat{M}}_{\mathfrak p}), K({\hat{R}}_{\mathfrak p}))= 0$ for all $i\neq h({\mathfrak p})=\dim(M_{\mathfrak p})- c.$
\item[(e)] $H^i_I(R)= 0$ for all $i\neq c,$ that is $I$ is a cohomologically complete intersection.
\end{itemize}
\end{thm}

Note that the natural homomorphisms, in $(b),(c)$, and $(d)$ of the above Theorem, can be constructed by truncation complex (see Definition \ref{ws} and Theorem \ref{4.2}).  Moreover $K({\hat{M}}_{\mathfrak p})$ denotes the canonical module of ${\hat{M}}_{\mathfrak p}$, see \cite{pet3} for its definition and basic properties.

In the case of $M = R$ a local Gorenstein ring the equivalence of the
conditions $(a),(b),(c),(d)$ was shown by Hellus and Schenzel as the
main result of their paper (see \cite[Theorem 0.1]{pet1}). The new point of view here
is the generalization to any maximal Cohen-Macaulay module of finite injective dimension. This implies a large bunch of new necessary conditions for an ideal to become set-theoretically a complete intersection.

As a byproduct of our investigations there is another characterization of a Cohen-Macaulay ring. H. Bass conjectured in his paper (see \cite{hum}) that if a Noetherian local ring $(R,\mathfrak m)$ possesses a non-zero finitely generated $R$-module $M$  such that its injective dimension $\id_R(M)$ is finite then $R$ is Cohen-Macaulay ring.  This was proved by M. Hochster (see \cite{ho}) in the equicharacteristic case and finally by P. Roberts (see \cite{r}). We add here the following characterization of a local Cohen-Macaulay ring.

\begin{thm}
Let $(R,\mathfrak m)$ be a Noetherian local ring, $I$ be an ideal of $R$, and $M\neq 0$ be a finitely generated $R$-module. Suppose that $H^i_I(M)= 0$ for all $i\neq c= \grade(I,M)$, then the following are equivalent:
\begin{itemize}
\item[(1)] $\id_R(M)< \infty$.
\item[(2)]  $\id_R(H^c_I(M))< \infty$.
\end{itemize}
If one of the equivalent conditions is satisfied then $R$ is a Cohen-Macaulay ring.
\end{thm}
Note that Zargar and Zakeri define a relative Cohen-Macaulay $R$-module $M\neq 0$ with respect to the ideal $I\subseteq R$ which is actually equivalent to the fact that $H^i_I(M)=0$ for all $i\neq c= \grade (I,M).$ They have shown that if $M$ is a relative Cohen-Macaulay $R$-module with respect to $I$, then $\id_R (H^c_I(M))=\id_R (M)- c$ (see \cite[Theorem 2.6]{z}). Moreover recently in \cite[Theorem 2.9]{z1} Zargar has proved the above Theorem \ref{4} incase of $M$ a maximal Cohen-Macaulay $R$-module which is relative Cohen-Macaulay in $\Supp_R(M/aM) \setminus \{\mathfrak{m}\}$.

The paper is organized as follows. In the Section 2 we recall a few preliminaries
used in the sequel of the paper. In Section 3 we make a comment to the Local Duality
for a Cohen-Macaulay ring. In Section 4 we describe the truncation complex as
it was introduced by Hellus and Schenzel  (see \cite[Definition 2.1]{pet1}) and use it for the proofs of our main results. In Section 5 we conclude with a few applications.

\section{Preliminaries}
In this section we will fix the notation of the paper and summarize a few preliminaries and auxiliary results.
We always assume that $(R,\mathfrak m)$ is a local commutative Noetherian ring with $\mathfrak m$ as a maximal ideal and $k= R/{\mathfrak m}$ denotes the residue field. Furthermore $E= E_R(k)$ denotes the injective hull of $k$.

Let $I\subseteq R$ be an ideal of $R$. For an $R$-module $M$ let $H^i_I(M)$, $i\in \mathbb Z$ denote the local cohomology modules of $M$ with respect to $I$ (see  \cite{goth} for its definition).
We define the grade and the cohomological dimension
\[
\grade(I,M)= \inf\{i\in \mathbb Z: H^i_I(M)\neq 0\} \text{ and }
\cd(I,M)= \sup\{i\in \mathbb Z: H^i_I(M)\neq 0\}.
\]
For a finitely generated $R$-module $M$ this notion of grade coincides with the usual one on the
maximal length of an $M$-regular sequence contained in $I$.
For any $R$-module $X$, $\id_R(X)$ stands for the injective dimension of $X$.

Moreover, we define $\height_M I = \height I R/\Ann_R M.$ Then
$\grade(I,M)\leq \height_M I \leq \cd(I,M).$ In the case of $M = R$
in addition it follows that $\height I \leq \ara I \leq \cd I.$ Furthermore
for a Cohen-Macaulay $R$-module $M$ it turns out that $\height_M I = \dim M -\dim M/IM$
for any ideal $I \subset R.$ For a Cohen-Macaulay $R$-module $M$ it is
clear that $\height_M I = \grade(I,M).$

\begin{rem}\label{1}
Let $\underline{x}= x_1, \ldots ,x_r\in I$ denote a system of elements of $R$ such that $\Rad I= \Rad(\underline{x})R$.
We consider the \v{C}ech complex $\Check{C}_{\underline{x}}$ of $\underline{x}$ with respect to $\underline{x}= x_1,...,x_r$. That is
\[
\Check{C}_{\underline{x}}= \mathop\otimes\limits_{i}^r \Check{C}_{x_i},
\]
where $\Check{C}_{x_i}$ is the complex $0\rightarrow R\rightarrow R_{x_i}\rightarrow 0$. Then $\Check{C}_{\underline{x}}$ has the following form
\[
0\rightarrow R\rightarrow \mathop\oplus\limits_{i=1}^r R_{x_i}\rightarrow ...\rightarrow R_{{x_1}{x_2}...{x_r}}\rightarrow 0.
\]
By $\Check{D}_{\underline{x}}$ we denote the truncation of $\Check{C}_{\underline{x}}$ by $R$. That is, we have  $\Check{D}_{\underline{x}}^i = \Check{C}_{\underline{x}}^i$ for all $i \not=0$ and $\Check{D}_{\underline{x}}^0 = 0.$

So there is a short exact sequence of complexes of flat $R$-modules
\[
0\rightarrow \Check{D}_{\underline{x}}\rightarrow \Check{C}_{\underline{x}}\rightarrow R\rightarrow 0,
\]
where $\Check{C}_{\underline{x}}\rightarrow R$ is the identity in homological degree zero.

For an arbitrary complex of $R$-module $X$ it follows (see \cite[Theorem 1.1]{pet2}) that
\[
H^i(\Check{C}_{\underline{x}}\otimes_R X)\cong H^i_I(X) \text{ for all }
i\in \mathbb Z.
\]
\end{rem}

Now let us summarize a few well-known facts about $\grade$ and  local cohomology. For basic notions on $\grade$ as well as other notions of commutative algebra we refer to Matsumura's textbook (see \cite{mat}). For the facts on homological algebra needed in this paper see \cite{w}. Also we denote ${\hat{R}}$ for the completion of $R$ with respect to the maximal ideal.

\begin{prop}\label{2.5}
For finitely generated $R$-modules $M$, $N$ and $I\subseteq R$ be an ideal we have
\begin{itemize}
\item[(a)] $\grade(I,M)= \inf\{\depth(M_{\mathfrak p}):$ ${\mathfrak p}\in \Supp_R(M)\cap V(I)\}$.
\item[(b)] Suppose that $\Supp_R(N)\subseteq \Supp_R(M)$, then $\cd(I,N)\leq \cd(I,M)$.
\end{itemize}
In particular $\cd(I,M)\leq \cd(I)$.
\end{prop}

\begin{proof}
The statement $(a)$ is shown in \cite[Proposition 1.2.10]{her}. For the proof of $(b)$ we refer to \cite[Theorem 2.2]{k}.
\end{proof}

In the context of our paper we are interested in Cohen-Macaulay rings and modules. A non-zero finitely generated $R$-module $M$ is called maximal Cohen-Macaulay module
if $\depth(M)= \dim(R).$

\begin{prop}\label{2.7}
Let $(R,\mathfrak m)$ be a Cohen-Macaulay ring of dimension $n$ and $I\subseteq R$ be an ideal. For a maximal Cohen-Macaulay module $M\neq 0$ we consider the following conditions:
\begin{itemize}
\item[(a)] $\grade(I)= \cd(I)$ and
\item[(b)] $\grade(I,M)= \cd(I,M)$.
\end{itemize}
Then $(a)$ implies $(b)$, while the converse is also true provided that $\Supp_R(M)= \Spec(R)$.
\end{prop}

\begin{proof}
Let $\mathfrak{p}\in \Supp_R(M)$. Because of $\depth(M_\mathfrak p)= \depth(R_\mathfrak p)= \dim(R_\mathfrak p)$ (recall that $M$ is maximal Cohen-Macaulay module) it follows that
\[
\grade(I)\leq \grade(I,M)\leq \cd(I,M)\leq \cd(I)
\]
So $(a)$ implies $(b)$. If $\Supp_R(M)= \Spec(R)$ we have that
$\grade(I)= \grade(I,M)$ and $\cd(I,M)= \cd(I)$ (see Proposition \ref{2.5}) and the converse holds also.
\end{proof}


\begin{notation}
As usual we use the symbol `` $\cong$ " in order to denote an isomorphism of modules. In contrast to that we use the symbol `` $\qism$ " in the following context:

Let $X\rightarrow Y$ be a morphism of complexes such that it induces an isomorphism in cohomologies, i.e. a quasi-isomorphism. Then we write $X \qism Y$. That is `` $\qism$ " indicates that there is a morphism of complexes in the right direction.

Moreover if $X \qism Y$ is a quasi-isomorphism and $F^{\cdot}_R$ is a complex of flat $R$-modules bounded above. Then it induces a quasi-isomorphism $F^{\cdot}_R\otimes_{R} X \qism F^{\cdot}_R\otimes_{R} Y$. Similar results are true for $\Hom_R(., E^{\cdot}_R)$ respectively $\Hom_R( P^{\cdot}_R, .)$ for $E^{\cdot}_R$ a bounded below complex of injective $R$-modules respectively $P^{\cdot}_R$ a bounded above complex of projective $R$-modules. For details we refer to \cite{har1}. This is the only fact we need of the theory of derived functors.
\end{notation}

\begin{lem}\label{4.3}
Let $(R,\mathfrak m)$ be a ring of dimension $n.$ Let $X$ be a complex of $R$-modules such that $\Supp_R(H^i(X))\subseteq V({\mathfrak m})$ for all $i\in \mathbb Z.$ Then
$H^i_{\mathfrak m} (X)\cong H^i(X)$ for all $i\in \mathbb Z$.
\end{lem}

\begin{proof}
Let $\underline{x}= x_1, \ldots ,x_n\in {\mathfrak m}$ denote a system of parameters of
$R$. Then we have the following short exact sequence of complexes of flat $R$-modules
\[
0\rightarrow \Check{D}_{\underline{x}}\rightarrow \Check{C}_{\underline{x}}\rightarrow R\rightarrow 0.
\]
Apply the functor $\cdot \otimes_{R} X$ to this sequence and we  get the following long exact sequence of cohomology modules
\[
\cdots \rightarrow H^i( \Check{D}_{\underline{x}}\otimes_{R} X)\rightarrow H^i_{\mathfrak m}(X)\rightarrow H^i(X)\rightarrow \cdots
\]
Now we claim that $\Check{D}_{\underline{x}}$$\otimes_{R} X$ is an exact complex.
This follows because
\[
\mathop\oplus \limits_{j}H^i (R_{x_j}\otimes_{R} X)\cong \mathop\oplus \limits_{j}R_{x_j}\otimes_{R} H^i(X)= 0
\]
This is true since $\Supp_R(H^i(X))\subseteq V({\mathfrak m})$ and cohomology commutes with exact functors. So it proves that
$H^i(X)\cong H^i_{\mathfrak m} (X)$ for all $i\in \mathbb Z,$ as required.
\end{proof}

In the following we need a result that was originally proved by Hellus and Schenzel (see \cite[Proposition 1.4]{pet1}) by a spectral sequence. Here we will give an elementary proof without using spectral sequences.

\begin{prop}\label{2.12}
Let $(R,\mathfrak m)$ be a local ring. Let $X$ be an arbitrary $R$-module. Then for any integer $s\in \mathbb N$ the following conditions are equivalent:
\begin{itemize}
\item[(1)] $H^i_{\mathfrak m} (X)= 0$ for all $i< s$.
\item[(2)] $\Ext^i_R(k,X)= 0$ for all $i< s$.
\end{itemize}
If one of the above conditions holds, then there is an isomorphism
\[
\Hom_R(k,H^s_{\mathfrak m} (X))\cong \Ext^s_R(k, X).
\]
\end{prop}

\begin{proof}
We prove the statement by an induction on $s$. First let us consider that $s= 0$.  Because of $\Supp(k)= \{{\mathfrak m}\}= V({\mathfrak m})$  the injection $\Gamma_{\mathfrak m}(X)\subseteq X$ induces an isomorphism
\[
\Hom_R(k,\Gamma_{\mathfrak m}(X))\cong \Hom_R(k, X).
\]
Because of $\Supp(\Gamma_{\mathfrak m}(X))\subseteq V({\mathfrak m})$ it follows that $\Gamma_{\mathfrak m}(X)= 0$ if and only if $\Hom_R(k,\Gamma_{\mathfrak m}(X))= 0$ which proves the claim for $s= 0$.

Now consider $s+1$ and assume that the statement is true  for all $i\leq s$. Since $\Supp(H^{s}_{\mathfrak m}(X))\subseteq V({\mathfrak m})$ it follows that the equivalence of the vanishing of $H^{i}_{\mathfrak m}(X)$ and $\Ext^{i}_R(k, X)$ for all $i\leq s$.

So it remains to prove that
\[
\Hom_R(k,H^{s+1}_{\mathfrak m} (X))\cong \Ext^{s+1}_R(k, X)
\]
To this end let $E^{\cdot}_R(X)$ be a minimal injective resolution of $X$, then
\[
\Hom_R(k, (E^{\cdot}_R(X))^i)= 0
\]
for all $i\leq s$. That is we have the following exact sequence
\[
0\rightarrow H^{s+1}_{\mathfrak m}(X)\rightarrow \Gamma_{\mathfrak m}(E^{\cdot}_R(X))^{s+1}\rightarrow \Gamma_{\mathfrak m}(E^{\cdot}_R(X))^{s+2}
\]
Because of  $\Hom_R(k, E^{\cdot}_R(X))\cong \Hom_R(k, \Gamma_{\mathfrak m}(E^{\cdot}_R(X)))$ it follows that $
\Ext^{s+1}_R(k, X)\cong \Hom_R(k,H^{s+1}_{\mathfrak m}(X))$.
\end{proof}

\begin{prop}\label{2.13}
Let $(R,\mathfrak m)$ be a local ring and $M\neq 0$ be a finitely generated $R$-module. Let
$I\subseteq R$ be an ideal of $R$ with $H^i_I(M)= 0$ for all
$i\neq c$ and $H^c_I(M)\neq 0$. Then $c= \grade(IR_\mathfrak p, M_\mathfrak p)$ for all ${\mathfrak p}\in V(I) \cap \Supp_R M$. If $M$ is in addition a Cohen-Macaulay module then $c= \height_{M_{\mathfrak p}} IR_{\mathfrak p}$ for all ${\mathfrak p}\in V(I) \cap \Supp_R M$.
\end{prop}

\begin{proof}  Since $c= \grade(I,M)\leq \grade(IR_\mathfrak p, M_\mathfrak p)$ for all ${\mathfrak p}\in V(I) \cap \Supp_R M$. Suppose that there is a prime ideal $\mathfrak p \in V(I)\cap \Supp_R M$ such that $\grade(IR_\mathfrak p, M_\mathfrak p)= h >c$. Then
\[
0 \not= H^h_{IR_{\mathfrak p}}(M_{\mathfrak p})\cong H^h_{IR_{\mathfrak p}}(M\otimes_{R} R_{\mathfrak p})\cong H^h_I(M)\otimes_{R} R_{\mathfrak p}.
\]
But it implies that $H^h_I(M)\neq 0$, $h> c,$ which is a contradiction to our assumption.
\end{proof}

In case of a not necessarily finitely generated $R$-module $X$ we put $\dim X = \dim \Supp_R X$. Note that this agrees with the usual notion $\dim X = \dim R/\Ann_R X$ if $X$ is finitely generated.

\section{Local Duality Theorem for a  Cohen-Macaulay ring}
We want to prove a variation of the Local Duality Theorem for a Cohen-Macaulay ring.
Let $(R,\mathfrak m)$ denote a local ring which is the factor ring of a
local Gorenstein ring $(S,{\mathfrak n})$ with $\dim(S)=t.$ Let $N$ be a finitely generated $R$-module.Then by the Local Duality Theorem  there is an isomorphism
\[
H^i_{\mathfrak m} (N)\cong \Hom_R(\Ext^{t-i}_S(N, S), E)
\]
for all $i\in \mathbb N$ (see \cite{goth}). Under these circumstances we define
\[
K(N):= \Ext^{t-r}_S(N, S), \dim(N)= r
\]
as the canonical module of $N$. It was introduced by Schenzel (see \cite{pet3}) as
the generalization of the canonical module of a ring (see e.g. \cite{her}).

For our purposes here we need an extension of the Local Duality of a local Gorenstein ring to a local  Cohen-Macaulay ring which is valid also for modules that are not necessarily finitely generated.
A more general result of Lemma \ref{3.1} was proved by Hellus (see \cite[Theorem 6.4.1]{he}). We
include here a proof of the particular case for sake of completeness.

\begin{lem}\label{3.1} \text{(\bf Local Duality})
Let $(R,\mathfrak m)$ be a Cohen-Macaulay ring of dimension $n$. Let $M$ be an arbitrary
$R$-module. Then there are functorial isomorphisms
\[
 \Hom_R(H^i_{\mathfrak m} (M), E)\cong \Ext^{n-i}_R(M, K({\hat{R}}))
\]
for all $i\in \mathbb N$.
\end{lem}

\begin{proof}
By Cohen's Structure Theorem any complete local ring $(R,\mathfrak m)$ is a homomorphic image of a regular local ring and - in particular - of local Gorenstein ring. Let ${\hat{R}}$ be the homomorphic image of a local Gorenstein ring $(S,{\mathfrak n})$ with $\dim(S)=t$. Then $H^n_{\mathfrak m} (R) \cong H^n_{{\mathfrak m}\hat{R}}(\hat{R})$ because any $R$-module $X$ with support contained in $V(\mathfrak m)$
admits the structure of an $\hat R$-module such that $X \cong X \otimes_R \hat R.$
Then by the Local Duality Theorem for Gorenstein rings (see \cite{goth}) there is an
isomorphism $\Hom_R(H^n_{\mathfrak m} (R), E) \cong \Ext^{t-n}_S({\hat{R}} , S).$

For a system of parameters $\underline{x}= x_1, \ldots ,x_n$ let $\Check{C}_{\underline{x}}$ denote the \v{C}ech complex with respect to $\underline{x}.$
Because $R$ is a Cohen-Macaulay ring $H^i_{\mathfrak m}(R) \cong H^i(\Check{C}_{\underline{x}}) = 0$ for all $i \not= n.$ That is, $\Check{C}_{\underline{x}}$ is
a flat resolution of $H^n_{\mathfrak m}(R)[-n].$ Therefore
\[
\Tor_{n-i}^R(M, H^n_{\mathfrak m}(R)) \cong H^i(\Check{C}_{\underline{x}} \otimes_R M) \cong H^i_{\mathfrak m}(M)
\]
for all $i \in \mathbb Z$ and an arbitrary $R$-module $M.$
Now we take the Matlis dual $\Hom_R(\cdot, E)$ of this isomorphism and get
\[
 \Hom_R(H^i_{\mathfrak m} (M), E)\cong \Ext^{n-i}_R(M, K({\hat{R}}))
\]
as it is easily seen by the adjunction isomorphism.
\end{proof}

Note that the above version of the Local Duality holds for an arbitrary $R$-module
$M.$ That is what we need in the sequel.

\begin{cor}
Let $(R,\mathfrak m)$ denote a Cohen-Macaulay ring of dimension $n$. Let $I\subseteq R$ be an ideal with $\grade(I, M)= c$ for an $R$-module $M$. Then there are isomorphisms
\[
\Ext^{n-i}_R(H^c_I(M), K({\hat{R}}))\cong \lim_{\longleftarrow} \Ext^{n-i}_R(\Ext^c_R(R/I^r, M),K({\hat{R}}))
\]
for all $i\in \mathbb N.$
\end{cor}

\begin{proof}
By Local Duality (Lemma \ref{3.1}), we have the following isomorphisms
\[
\Ext^{n-i}_R(H^c_I(M), K({\hat{R}}))\cong \Hom_R(H^i_{\mathfrak m} (H^c_I(M)), E)
\]
for all $i \in \mathbb Z.$ The module on the right side is isomorphic to
\[
\Hom_R( \lim_{\longrightarrow} H^i_{\mathfrak m} (\Ext^c_R(R/I^r, M)), E)
\]
since $H^c_{I}(M) \cong \varinjlim \Ext^c_R(R/I^r,M)$ and because the
local cohomology commutes with direct limits. Finally
\[
\Ext^{n-i}_R(H^c_I(M), K({\hat{R}})) \cong
\lim_{\longleftarrow} \Hom_R(H^i_{\mathfrak m} (\Ext^c_R(R/I^r, M)), E)
\]
since the $\Hom$-functor in the first variable transforms a direct system into an inverse system, see \cite[Proposition 6.4.6]{v}. Then the claim follows since
\[
\Hom_R(H^i_{\mathfrak m} (\Ext^c_R(R/I^r, M)), E)
\cong \Ext^{n-i}_R(\Ext^c_R(R/I^r, M), K({\hat{R}}))
\]
as it is true again by Local Duality (see Lemma \ref{3.1}).
\end{proof}

\section{The Truncation Complex}
Let $(R,\mathfrak m)$ be a local ring of dimension $n$. Let $M\neq 0$ denote a finitely generated  $R$-module and $\dim M = n.$
Let $I\subseteq R$
be an ideal of $R$ with $\grade(I,M) = c.$ Suppose that $E^{\cdot}_R(M)$ is a minimal injective resolution of $M.$ Then it follows from
(Matlis \cite{m} or Gabriel \cite{g}) that
\[
E^{\cdot}_R(M)\cong\bigoplus\limits_{{\mathfrak p}\in \Supp M}E_R(R/{\mathfrak p})^{\mu_{i}({\mathfrak p}, M)},
\]
where $\mu_{i}({\mathfrak p}, M)=\dim_{k({\mathfrak p})}(\Ext^i_{R_{\mathfrak p}}(k({\mathfrak p}),M_{\mathfrak p}))$.
We get $\Gamma_{I}(E_R(R/{\mathfrak p}))=0$ for all $\mathfrak p\notin V(I)$ and $\Gamma_{I}(E_R(R/{\mathfrak p}))= E_R(R/{\mathfrak p})$ for all $\mathfrak p\in V(I)$.
Moreover $\mu_{i}({\mathfrak p}, M)= 0$ for all $i< c$ since $\grade (I,M) =c$.
Whence for all $i< c$ it follows that
\[
\Gamma_I(E^{\cdot}_R(M))^{i}\cong\bigoplus\limits_{{\mathfrak p}\in V(I) \cap \Supp M} \Gamma_I (E_R(R/{\mathfrak p}))^{\mu_{i}({\mathfrak p}, M)}= 0.
\]
Therefore $H^c_I(M)$ is isomorphic to the kernel of $\Gamma_I(E^{\cdot}_R(M))^c\rightarrow \Gamma_I(E^{\cdot}_R(M))^{c+1}.$ Whence there is an embedding of complexes of $R$-modules  $ H^c_I(M)[-c]\rightarrow \Gamma_I(E^{\cdot}_R(M))$.

\begin{defn}\label{ws}
Let $C^{\cdot}_M(I)$ be the cokernel of the above embedding. It is called the truncation
complex. So there is a short exact sequence of complexes of $R$-modules
\[
0\rightarrow H^c_I(M)[-c]\rightarrow \Gamma_I(E^{\cdot}_R(M))\rightarrow C^{\cdot}_M(I)\rightarrow 0.
\]
In particular it follows that $H^i(C^{\cdot}_M(I))= 0$ for all $i\leq c$ or $i> n$ and $H^i(C^{\cdot}_M(I))\cong H^i_I(M)$ for all $c< i\leq n.$
\end{defn}

The advantage of the truncation complex is that it separates the information of cohomology modules of $H^i_I(M)$ for $i= c$ from those with $i\neq c$.

\begin{thm}\label{4.2}
Fix the previous notation. Let $M$ be a maximal Cohen-Macaulay module with $\dim M/IM = d.$ We put $c = \dim M - \dim M/IM = \grade(I,M).$ Then we have the following results:
\begin{itemize}
\item[(a)] There are an exact sequence
\[
0\rightarrow H^{n-1}_{\mathfrak m}(C^{\cdot}_M(I))\rightarrow H^d_{\mathfrak m}(H^c_I(M))\rightarrow H^n_{\mathfrak m} (M)\rightarrow H^{n}_{\mathfrak m}(C^{\cdot}_M(I))\rightarrow 0
\]
and isomorphisms $H^{i-1}_{\mathfrak m}(C^{\cdot}_M(I))\cong H^{i-c}_{\mathfrak m}(H^c_I(M))$ for all $i\neq n,n+1$.
\item[(b)] There are an exact sequence
\[
0\rightarrow \Ext^{n-1}_R(k, C^{\cdot}_M(I))\rightarrow \Ext^{d}_R(k,H^c_I(M))\rightarrow \Ext^{n}_R(k, M)\rightarrow
\Ext^{n}_R(k, C^{\cdot}_M(I))
\]
and isomorphisms $\Ext^{i-c}_R(k, H^c_I(M))\cong \Ext^{i-1}_R(k, C^{\cdot}_M(I))$ for all $i< n.$
\item[(c)] Assume in addition that $R$ is a Cohen-Macaulay ring. There are an exact sequence
\[
0\rightarrow \Ext^{0}_R(C^{\cdot}_M(I), K({\hat{R}}))\rightarrow K({\hat{M}})\rightarrow \Ext^{c}_R(H^c_I(M), K({\hat{R}}))\rightarrow \Ext^{1}_R(C^{\cdot}_M(I), K({\hat{R}}))\rightarrow 0
\]
and isomorphisms $\Ext^{i+c}_R(H^c_I(M), K({\hat{R}}))\cong \Ext^{i+1}_R(C^{\cdot}_M(I), K({\hat{R}}))$ for all $i> 0.$
\end{itemize}
\end{thm}

\begin{proof}
$(a)$ Let $\underline{x}= x_1, \ldots,x_n\in {\mathfrak m}$ denote a system of parameters of $R$. We tensor the short exact sequence of the truncation complex by the \v{C}ech complex $\Check{C}_{\underline{x}}.$ Then the resulting sequence of complexes remains exact because $\Check{C}_{\underline{x}}$ is a complex of flat $R$-modules. That is, there is the following short exact sequence of complexes of $R$-modules
\[
0\rightarrow \Check{C}_{\underline{x}}\otimes_{R} H^c_I(M)[-c]\rightarrow \Check{C}_{\underline{x}}\otimes_{R} \Gamma_I(E^{\cdot}_R(M))\rightarrow \Check{C}_{\underline{x}}\otimes_{R} C^{\cdot}_M(I)\rightarrow 0.
\]
Now we look at the cohomology of the complex in the middle. Since $\Gamma_I(E^{\cdot}_R(M))$ is a complex of injective $R$-modules the natural morphism
\[
\Gamma_{\mathfrak m}(E^{\cdot}_R(M)) \cong \Gamma_{\mathfrak m}(\Gamma_I(E^{\cdot}_R(M))) \to \Check{C}_{\underline{x}}\otimes_{R} \Gamma_I(E^{\cdot}_R(M))
\]
induces an isomorphism in cohomology (see \cite[Theorem 1.1]{pet2}). Because $M$ is a maximal Cohen-Macaulay
$R$-module the only non-vanishing local cohomology module is $H^n_{\mathfrak m}(M).$
So the result follows from the long exact cohomology sequence.

$(b)$ Let $F^{\cdot}_R(k)$ be a free resolution of $k$.
Apply the functor $\Hom_R(F^{\cdot}_R(k), .)$ to the short exact sequence of the truncation complex. Then it induces the following short exact sequences of complexes of $R$-modules
\[
0\rightarrow \Hom_R(F^{\cdot}_R(k), H^c_I(M))[-c]\rightarrow \Hom_R(F^{\cdot}_R(k), \Gamma_I(E^{\cdot}_R(M)))\rightarrow \Hom_R(F^{\cdot}_R(k), C^{\cdot}_M(I))\rightarrow 0.
\]
Now $\Hom_R(F^{\cdot}_R(k), \Gamma_I(E^{\cdot}_R(M))) \qism \Gamma_I(\Hom_R(F^{\cdot}_R(k), E^{\cdot}_R(M)))$ since $F^{\cdot}_R(k)$ is a right bounded complex of finitely generated free $R$-modules. Since any $R$-module of the complex $\Hom_R(F^{\cdot}_R(k), E^{\cdot}_R(M))$ is injective. By \cite[Theorem 1.1]{pet2} it implies that the complex $\Gamma_I(\Hom_R(F^{\cdot}_R(k), E^{\cdot}_R(M)))$ is quasi-isomorphic to $\Check{C}_{\underline{y}}\otimes_{R} \Hom_R(F^{\cdot}_R(k), E^{\cdot}_R(M))$ where $\underline{y}= y_1, \ldots ,y_r\in I$ such that $\Rad I= \Rad(\underline{y})R$.

Now tensoring with a right bounded complex of flat $R$-modules preserves the quasi-isomorphisms and any $R$-module of $E^{\cdot}_R(M)$ is injective. It induces the following quasi-isomorphism
\[
\Check{C}_{\underline{y}}\otimes_{R} \Hom_R(k, E^{\cdot}_R(M)) \qism \Check{C}_{\underline{y}}\otimes_{R} \Hom_R(F^{\cdot}_R(k), E^{\cdot}_R(M)).
\]
But the complex on the left side is isomorphic to $\Hom_R(k, E^{\cdot}_R(M))$. This is true because of each $R$-module of the complex $\Hom_R(k, E^{\cdot}_R(M))$ has support in $V(\mathfrak{m})$. Therefore the complex $\Hom_R(F^{\cdot}_R(k), \Gamma_I(E^{\cdot}_R(M)))$ is quasi-isomorphic to $\Hom_R(k, E^{\cdot}_R(M))$. Then the result follows by the long exact cohomology sequence. To this end note that the assumption on the maximal Cohen-Macaulayness of $M$ implies that $\Ext^{i}_R(k, M)= 0$ for all $i< n.$

(c) Let $E^{\cdot}_R(K({\hat{R}}))$ be a minimal injective resolution of $K({\hat{R}})$. We
apply the functor $\Hom_R( ., E^{\cdot}_R(K({\hat{R}})))$ to the short exact sequence of the truncation complex. Then we have a short exact sequence of complexes of $R$-modules
\begin{gather*}
0\rightarrow \Hom_R(C^{\cdot}_M(I), E^{\cdot}_R(K({\hat{R}})))\rightarrow \Hom_R(\Gamma_I(E^{\cdot}_R(M)),
E^{\cdot}_R(K({\hat{R}})))\\
\rightarrow \Hom_R(H^c_I(M), E^{\cdot}_R(K({\hat{R}})))[c]\rightarrow 0.
\end{gather*}
By the definition of the section functor $\Gamma_I(\cdot)$ and a standard isomorphisms
for the direct and inverse limits  the complex in the middle is isomorphic to
\[
\lim_{\longleftarrow}
\Hom_R(\Hom_R(R/I^r,E^{\cdot}_R(M)), E^{\cdot}_R(K({\hat{R}})))
\cong \lim_{\longleftarrow} (R/I^r\otimes_{R} \Hom_R(E^{\cdot}_R(M),
E^{\cdot}_R(K({\hat{R}}))))
\]
For the last isomorphism note that $R/I^r$ is finitely generated $R$-module for all $r \geq 1.$
We claim that $X := \Hom_R(E^{\cdot}_R(M), E^{\cdot}_R(K({\hat{R}})))$
is a flat resolution of $K({\hat{M}})$. To
this end we first note that
\[
X^s = \prod_{i \in \mathbb{Z}}
\Hom_R(E^i_R(M), E^{i+s}_R(K({\hat{R}}))).
\]
Therefore $X^s, s \in \mathbb{Z},$ is a flat $R$-module. Moreover, by the
definition we have that $H^s(X) \cong \Ext^s_R(M,K(\hat{R}))$ for
all $s\in \mathbb{Z}.$ Since $M$ is a Cohen-Macaulay $R$-module and
$R$ is a Cohen-Macaulay ring it follows by the Local Duality Theorem (see Lemma \ref{3.1})
that $H^s(X) = 0$ for all $s \not= 0$ and $H^s(X) \cong K({\hat{M}})$
for $s = 0.$ In order to prove the claim we have to show that $X^s = 0$
for all $s > 0$. Since $R$ is a Cohen-Macaulay ring so
\[
E^{i}_R(K({\hat{R}}))) \cong
\oplus_{\height{p} =i}E_R(R/\mathfrak{p})
\]
(see \cite{sharp}).
On the other side $E^{i+s}_R(M) \cong \oplus_{\mathfrak{q}\in \Supp_R (M)}
E_R(R/\mathfrak{q})^{\mu_{i}(\mathfrak{q},M)}.$ Since $M$ is a maximal Cohen-Macaulay $R$-module it follows that $\mu_{i}(\mathfrak{q},M)= 0$ for all ${\mathfrak{q}\in \Supp_R (M)}$ such that $i+s< \height(\mathfrak{q})$. Now let $i+s\geq \height(\mathfrak{q})$ and $\height(\mathfrak{p})= i$. Then we have that $\Hom_R(E_R(R/\mathfrak{q}), E_R(R/\mathfrak{p}))= 0$ for all $k> 0.$ Moreover this implies that
\[
\Hom_R(E_R(R/\mathfrak{q}), E^{i}_R(K({\hat{R}})))= 0
\]
for all $i$ as easily seen. This finally implies that $X^s= 0$ for all $s> 0$. So $X$ is indeed a flat resolution of $K({\hat{M}})$.
Then the cohomologies of the complex
\[
\lim_{\longleftarrow} (R/I^r\otimes_{R} \Hom_R(E^{\cdot}_R(M), E^{\cdot}_R(K({\hat{R}}))))
\]
are zero for all $i\not=0$ and
$L_0\Lambda^I(K({\hat{M}}))= K({\hat{M}})$.
So the long exact sequence of local cohomologies provides the statements in (c).
\end{proof}

\subsection{Necessary Condition of $H^i_I(M)= 0$, for all $i\neq c$}
In the following let $(R,\mathfrak m)$ denote a local ring of dimension $n$. Let
$M\neq 0$ be a finitely generated $R$-module with $\dim M = n.$ Let $I \subset R$ be an
ideal such that $c = \grade(I,M).$

\begin{cor}\label{4.4}
Let $M\neq 0$ be a maximal Cohen-Macaulay $R$-module such that $H^i_I(M)= 0$ for all $i\neq c = \grade(I,M).$ Then:
\begin{itemize}
\item[(a)] The natural homomorphism
$$H^d_{\mathfrak m} (H^c_I(M))\rightarrow H^n_{\mathfrak m} (M)$$
is an isomorphism and $H^i_{\mathfrak m} (H^c_I(M))= 0$ for all $i\neq d.$
\item[(b)] The natural homomorphism
$$\Ext^{d}_R(k,H^c_I(M))\rightarrow \Ext^{n}_R(k, M)$$
is an isomorphism and $\Ext^i_R(k,H^c_I(M))=0$ for all $i<d$.
\item[(c)] Suppose that $R$ is in addition a Cohen-Macaulay ring. The natural homomorphism
$$K({\hat{M}})\rightarrow \Ext^{c}_R(H^c_I(M), K({\hat{R}}))$$
is an isomorphism and $\Ext^{i+c}_R(H^c_I(M), K({\hat{R}}))= 0$ for all $i> 0.$
\end{itemize}
\end{cor}

\begin{proof}
Because of assumption $C^{\cdot}_M(I)$ is an exact complex.

First we prove part $(a)$. Apply the $\Check{C}$ech complex $\Check{C}_{\underline{x}}\otimes_{R} \cdot$
to the short exact sequence of the truncation complex since $\Check{C}_{\underline{x}}$ is a complex of flat $R$- modules so $\Check{C}_{\underline{x}}\otimes_{R} C^{\cdot}_M(I)$ is exact. Whence result follows from Theorem \ref{4.2}.

Now we prove $(b).$
Let $F^{\cdot}_R(k)$ be a free resolution of $k.$ Then apply $\Hom_R(F^{\cdot}_R(k), .)$ to the short exact sequence of the truncation complex. Since $\Hom_R(F^{\cdot}_R(k), C^{\cdot}_M(I))$ is an exact complex the result follows from Theorem \ref{4.2}.

Finally we prove $(c).$
Let $E^{\cdot}_R(K({\hat{R}}))$ be injective resolution of $K({\hat{R}}).$ Then apply $\Hom_R(.,  E^{\cdot}_R(K({\hat{R}})))$ to the short exact sequence of the truncation complex. Since $\Hom_R(C^{\cdot}_M(I)),E^{\cdot}_R(K({\hat{R}})))$ is an exact complex the result follows from Theorem \ref{4.2}.
\end{proof}

\begin{lem}\label{4.6}
Let $(R,\mathfrak m)$ be ring, $I\subseteq R$ be an ideal, and $M\neq 0$ be a maximal Cohen-Macaulay $R$-module. Suppose that $H^i_I(M)= 0$ for all $i\neq c = \grade(I,M)$ and ${\mathfrak p}\in \Supp_R(M)\cap V(I).$
 Then
\begin{itemize}
\item[(1)] $H^i_{IR_{\mathfrak p}}(M_{\mathfrak p})= 0$ for all $i\neq c$.
\item[(2)] The natural homomorphism
$$H^i_{{\mathfrak p}R_{\mathfrak p}}(H^c_{IR_{\mathfrak p}}(M_{\mathfrak p}))\rightarrow H^{\dim(M_{\mathfrak p})}_{{\mathfrak p}R_{\mathfrak p}}(M_{\mathfrak p})$$
is an isomorphism for $i= h({\mathfrak p})= \dim(M_{\mathfrak p})- c$ and $H^i_{{\mathfrak p}R_{\mathfrak p}}(H^c_{IR_{\mathfrak p}}(M_{\mathfrak p}))= 0$ for all $i\neq h({\mathfrak p})= \dim(M_{\mathfrak p})- c$.
\end{itemize}
\end{lem}

\begin{proof}
It is follow from Proposition \ref{2.13} and Corollary \ref{4.4} $(a)$.
\end{proof}
\subsection{Converse(Sufficient Condition of $H^i_I(M)= 0$, for all $i\neq c$)}
Before we can prove Theorem \ref{4} which is one of the main result of this section we need the following Lemma. For the following technical result we need a few details on derived categories
and derived functors. For all these facts we refer to the Lecture Note by
R. Hartshorne (see \cite{har1}). We are grateful to P. Schenzel for suggesting this argument to us.

\begin{lem}\label{w4}
Let $M\neq 0$ be an $R$-module such that $\dim_R(M)=n=\dim(R)$. Let $c= \grade(I,M)$ where $I\subseteq R$ is an ideal. Then the following are equivalent:
\begin{itemize}
\item[(1)] The natural homomorphism
\[
\Ext^{i-c}_R(k,H^c_I(M))\rightarrow \Ext^{i}_R(k, M)
\]
is an isomorphism for all $i \in \mathbb Z$.
\item[(2)] The natural homomorphism
\[
H^{i-c}_{\mathfrak m} (H^c_I(M))\rightarrow H^i_{\mathfrak m} (M)
\]
is an isomorphism for all $i \leq n$.
\end{itemize}
\end{lem}

\begin{proof}
First of all note that the statement in $(2)$ is equivalent
to the isomorphism for all $i \in \mathbb Z$. This follows because of
$\dim_R (H^c_I(M)) \leq d= n-c$ and $\dim_R (M) = n$ (see \cite{k}). By applying ${\rm{R}} \Gamma_\mathfrak{m}(\cdot)$ to the short exact sequence of the truncation complex it follows that the assumption in (2) is equivalent to the fact that ${\rm{R}} \Gamma_\mathfrak{m}(C^{\cdot}_M(I))$ is an exact complex.

Moreover by applying the derived functor ${\rm{R}} \Hom(k, \cdot)$ to the short exact
sequence of the truncation complex it follows that the statement in (1) is
equivalent to the fact that ${\rm{R}} \Hom(k, C^{\cdot}_M(I))$ is an exact
complex.

Now we prove that (2) implies (1). If ${\rm{R}} \Gamma_\mathfrak{m}(C^{\cdot}_M(I))$ is an
exact complex the same is true for
\[
{\rm{R}} \Hom(k, C^{\cdot}_M(I)) \cong {\rm{R}} \Hom(k,{\rm{R}} \Gamma_\mathfrak{m}(C^{\cdot}_M(I)))
\]
since $C^{\cdot}_M(I)$ is a left bounded complex.

In order to prove that (1) implies (2) consider the short exact sequence (in fact an exact triangle in
the derived category)
\[
0 \to \mathfrak{m}^{\alpha}/\mathfrak{m}^{\alpha +1 } \to R/\mathfrak{m}^{\alpha +1}
\to R/\mathfrak{m}^{\alpha} \to 0.
\]
for $\alpha \in \mathbb{N}$. By applying the derived functor ${\rm{R}} \Hom(\cdot,C^{\cdot}_M(I))$ it provides a short
exact sequence of complexes
\[
0 \to {\rm{R}} \Hom(R/\mathfrak{m}^{\alpha},C^{\cdot}_M(I)) \to
{\rm{R}} \Hom(R/\mathfrak{m}^{\alpha +1},C^{\cdot}_M(I))
\to {\rm{R}} \Hom(\mathfrak{m}^{\alpha}/\mathfrak{m}^{\alpha +1 } ,C^{\cdot}_M(I)) \to 0.
\]
By induction on $\alpha$ it follows that ${\rm{R}} \Hom(R/\mathfrak{m}^{\alpha},C^{\cdot}_M(I))$
is an exact complex for all $\alpha \in \mathbb{N}$. Since $\varinjlim {\rm{R}}
 \Hom(R/\mathfrak{m}^{\alpha},C^{\cdot}_M(I)) \cong {\rm{R}} \Gamma_\mathfrak{m}(C^{\cdot}_M(I))$ it follows that
 ${\rm{R}} \Gamma_\mathfrak{m}(C^{\cdot}_M(I))$ is exact. This finishes the proof of the Lemma.

\end{proof}

\begin{cor}\label{4.x}
Let $M\neq 0$ be a maximal Cohen-Macaulay module over a Cohen-Macaulay ring $R$ of dimension $n$ and $I\subseteq R$ be an ideal of $c = \grade(I,M).$  Then the following are equivalent:
\begin{itemize}
\item[(1)] The natural homomorphism
\[
\Ext^{i-c}_R(k,H^c_I(M))\rightarrow \Ext^{i}_R(k, M)
\]
is an isomorphism for $i \geq n$ and $\Ext^i_R(k,H^c_I(M))=0$ for all $i< d=n-c$.
\item[(2)] The natural homomorphism
\[
H^d_{\mathfrak m} (H^c_I(M))\rightarrow H^n_{\mathfrak m} (M)
\]
is an isomorphism and $H^i_{\mathfrak m} (H^c_I(M))= 0$ for all $i\neq d$ where $d= n-c$.
\end{itemize}
\end{cor}
\begin{proof} This is an immediate consequence of Lemma \ref{w4}.
\end{proof}

 In case we have $\id_R (M) < \infty$ for the $R$-module $M$ in Corollary \ref{4.x} we get the
 following Corollary. Note that this Corollary provides us a new relation between the canonical module and the Ext module.
\begin{cor}\label{4.5}
Let $M\neq 0$ be a maximal Cohen-Macaulay module of finite injective dimension over a local ring $R$ of dimension $n$ and $I\subseteq R$ be an ideal of $c = \grade(I,M).$  Then the following are equivalent:
\begin{itemize}
\item[(1)] The natural homomorphism
\[
\Ext^{d}_R(k,H^c_I(M))\rightarrow \Ext^{n}_R(k, M)
\]
is an isomorphism and $\Ext^i_R(k,H^c_I(M))=0$ for all $i\neq d=n-c$.
\item[(2)] The natural homomorphism
\[
H^d_{\mathfrak m} (H^c_I(M))\rightarrow H^n_{\mathfrak m} (M)
\]
is an isomorphism and $H^i_{\mathfrak m} (H^c_I(M))= 0$ for all $i\neq d$ where $d= n-c$.
\item[(3)] The natural homomorphism
\[
K({\hat{M}})\rightarrow \Ext^{c}_R(H^c_I(M), K({\hat{R}}))
\]
is an isomorphism and $\Ext^{n-i}_R(H^c_I(M), K({\hat{R}}))= 0$ for all $i\neq d$.
\end{itemize}
\end{cor}
\begin{proof}
Since $M$ is of finite injective dimension so by \cite{r} it follows that $R$ is Cohen-Macaulay ring. Note that the equivalence of $(1)\Leftrightarrow (2)$ follows form Corollary \ref{4.x}. Recall that
$\Ext_R^i(k,M) = 0$ for all $i \not= n$ since $\id_R (M) = n$ under the assumption.

Next we proof the equivalence of $(2)\Leftrightarrow (3).$

In fact this is a consequence of the generalized Local Duality (see Lemma \ref{3.1}) and Matlis Duality.
\end{proof}
\begin{rem} \label{rem}
$(1)$ Let $M\neq 0$ be a maximal Cohen-Macaulay module of finite injective dimension over a local ring $R$. Suppose that $I\subseteq R$ is an ideal with $c = \grade(I,M).$ If any of the equivalent conditions holds of Corollary \ref{4.5} then it follows that all the Bass numbers of $H^c_I(R)$ are zero except for $i=d=\dim_R(M)-c.$ \\
$(2)$ Let $R$ be a complete local ring of dimension $n$ and $M\neq 0$ be a maximal Cohen-Macaulay module of finite injective dimension. Let $I\subseteq R$ be an ideal of
$c = \grade(I,M)$. Then $K(R)$ exists and by \cite[Theorem 4.1]{sharp} it follow that $M\cong \oplus K(R)$. Therefore if $R$ is complete so it is enough to prove Corollary \ref{4.5} for $K(R)$ instead of $M$.\\
$(3)$ In case $R$ possesses a maximal Cohen-Macaulay module of finite injective dimension $M\not=0$ the ring $R$ is Cohen-Macaulay and it follows that $\Supp_R M = \Spec R$. To this end first note that
$\Supp_{\hat{R}} \hat{M} = \Spec \hat{R}$ since $\hat{M}$ is isomorphic to a direct sum of $K(\hat{R})$
and $\Ann K(\hat{R}) = (0)$. This implies $\Ann_R M = 0$ which implies the claim.
\end{rem}
Now we can prove one of our main result as follows:
\begin{thm}\label{4}
Let $M\neq 0$ be a maximal Cohen-Macaulay module of $\id_R(M)< \infty$ over a local ring $R$ of dimension $n$ and $I\subseteq R$ be an ideal of $c = \grade(I,M).$ The following
conditions are equivalent:
\begin{itemize}
\item[(a)] $H^i_I(M)= 0$ for all $i\neq c.$
\item[(b)] For all ${\mathfrak p}\in  V(I)\cap \Supp_R(M)$ the natural homomorphism
\[
H^{h({\mathfrak p})}_{{\mathfrak p}R_{\mathfrak p}}(H^c_{IR_{\mathfrak p}}(M_{\mathfrak p})) \to H^{\dim(M_{\mathfrak p})}_{{\mathfrak p}R_{\mathfrak p}}(M_{\mathfrak p})
\]
is an isomorphism and $H^i_{{\mathfrak p}R_{\mathfrak p}}(H^c_{IR_{\mathfrak p}}(M_{\mathfrak p}))= 0$ for all $i\neq h({\mathfrak p})=\dim(M_{\mathfrak p})- c .$
\item[(c)] For all ${\mathfrak p}\in  V(I)\cap \Supp_R(M)$ the natural homomorphism

$$\Ext^{h({\mathfrak p})}_{R_{\mathfrak p}}(k({\mathfrak p}),H^c_{IR_{\mathfrak p}}(M_{\mathfrak p}))\rightarrow \Ext^{\dim(M_{\mathfrak p})}_{R_{\mathfrak p}}(k({\mathfrak p}), (M_{\mathfrak p}))$$
is an isomorphism and $\Ext^i_{R_{\mathfrak p}}(k({\mathfrak p}), H^c_{IR_{\mathfrak p}}(M_{\mathfrak p}))= 0$ for all $i\neq h({\mathfrak p})= \dim(M_{\mathfrak p})- c$.
\item[(d)] For all ${\mathfrak p}\in V(I)\cap \Supp_R(M)$ the natural homomorphism
\[
K({\hat{M}}_{\mathfrak p})\rightarrow \Ext^{c}_{{\hat{R}}_{\mathfrak p}}(H^c_{I{{\hat{R}}_{\mathfrak p}}}({\hat{M}}_{\mathfrak p}), K({\hat{R}}_{\mathfrak p}))
\]
is an isomorphism and $\Ext^{i}_{{\hat{R}}_{\mathfrak p}}(H^c_{I{{\hat{R}}_{\mathfrak p}}}({\hat{M}}_{\mathfrak p}), K({\hat{R}}_{\mathfrak p}))= 0$ for all $i\neq h({\mathfrak p}).$
\item[(e)] $H^i_I(R)= 0$ for all $i\neq c,$ that is $I$ is a cohomologically complete intersection.
\end{itemize}

\end{thm}
\begin{proof} Firstly we prove the equivalence of (a) and (e).  By Remark \ref{rem} the
assumptions imply that $\Supp_R M = \Spec R$ and the equivalence follows
by Proposition \ref{2.7}.

For the proof that $(a)$ implies $(b)$ see Lemma \ref{4.6}. By Corollary \ref{4.5} the equivalence of $(b),(c),(d)$ is easily seen by passing to the localization.

Now we prove that $(b)$ implies $(a)$

We proceed by induction on $\dim(M/IM)$. If $\dim(M/IM)= 0$ then $V(I)\cap \Supp_R(M)
\subseteq V(\mathfrak m)$ so statement $(b)$ holds for
$\mathfrak p= \mathfrak m$. By the definition of the truncation complex it implies that $H^i_\mathfrak m(C^{\cdot}_M(I))= 0$ for all $i\in \mathbb Z.$

Now we apply Lemma \ref{4.3} for $X= C^{\cdot}_M(I)$. Because of $\Supp_R(H^i(C^{\cdot}_M(I)))\subseteq V(\mathfrak m)$ (recall that $H^i(C^{\cdot}_M(I))\cong H^i_I(M)$ for all $i\neq c$ and $H^c(C^{\cdot}_M(I))= 0$ (see the definition 4.1)). This proves the vanishing of $H^i_I(M)$ for all $i\neq c$.

Now let $\dim(M/IM)> 0$ then
\[
\dim(M_{\mathfrak p}/IM_{\mathfrak p})< \dim(M/IM)
\]
for all ${\mathfrak p} \in V(I)\cap \Supp_R(M) \setminus \{ \mathfrak m\}$. By the induction hypothesis it implies that
\[
H^i_{IR_{\mathfrak p}}(M_{\mathfrak p})= 0
\]
for all $i\neq c$ and all ${\mathfrak p} \in V(I)\cap \Supp(M)\setminus \{ \mathfrak m\}$. That is $\Supp(H^i_I(M))\subseteq V({\mathfrak m})$ for all $i\neq c$. It implies that $H^i_{\mathfrak m}(C^{\cdot}_M(I)) = 0 $ for $i \leq c$ and $i > n$. Also for $c < i \leq n$ there is an isomorphism
\[
 H^i_{\mathfrak m}(C^{\cdot}_M(I)) \cong H^i(C^{\cdot}_M(I)) \cong H^i_I(M).
\]
On the other side note that by the assumption for $\mathfrak{p}= \mathfrak{m}$ we have $H^i_{\mathfrak m}(C^{\cdot}_M(I))= 0$ for all $i\in \mathbb{Z}$ (by Theorem \ref{4.2}) and hence by virtue of the last isomorphism it gives the result. This completes the proof of the Theorem.
\end{proof}

\begin{rem}
Let us discuss the necessity of the local conditions in Theorem \ref{4}. It is not enough to assume the statements $(b)$, $(c)$, $(d)$  in the Theorem \ref{4} for ${\mathfrak p}= {\mathfrak m}$. That is, we do need these statements for all ${\mathfrak p} \in V(I)\cap \Supp(M)$. This is shown by Hellus and Schenzel (see \cite[Example 4.1]{pet1}) where $(b)$ is true for ${\mathfrak p}= {\mathfrak m}$ but not for a localization at ${\mathfrak p}$. That is, these three properties $(b)$, $(c)$, $(d)$ do not localize.
\end{rem}

In the next we are interested in the injective dimension of the local cohomology module. As a consequence it provides a new characterization of a local ring to be Cohen-Macaulay.
\begin{thm}\label{9}
Let $(R,\mathfrak m)$ be a ring and $I$ be an ideal of $R$. Suppose that $M \not= 0$ is a finitely generated $R$-module such that $H^i_I(M)= 0$ for all $i\neq c= \grade(I,M)$ then the following are equivalent:
\begin{itemize}
\item[(1)] $\id_R(M)< \infty$.
\item[(2)]  $\id_R(H^c_I(M))< \infty$.
\end{itemize}
Each of the equivalent conditions implies that $R$ is a Cohen-Macaulay ring.
\end{thm}
\begin{proof}
Let $E^{\cdot}_R(M)$ be a minimal injective resolution of $M$. Since $H^i_I(M)= 0$ for all $i\neq c$ so $H^c_I(M)[-c]\to \Gamma_I(E^{\cdot}_R(M))$ is an isomorphism of complexes in cohomology. If $\id_R(M)< \infty$, then $\Gamma_I(E^{\cdot}_R(M))$ is a finite resolution of $H^c_I(M)[-c]$ by injective $R$-modules. This proves $(1) \Rightarrow (2)$.

For the converse statement note that the above quasi-isomorphism of complexes induces the following isomorphism
\[
\Ext^{i-c}_R(k,H^c_I(M))\cong \Ext^i_R(k,M)
\]
so it follows that $\id_R(M)< \infty$ since $\Ext^{i-c}_R(k,H^c_I(M))= 0$ for $i\gg 0$.

Note that if one of the equivalent conditions in Theorem \ref{9} hold then $R$ will be a Cohen-Macaulay  ring (see \cite{r}).
\end{proof}

\section{Applications}

In this section we will give some applications of our main results one of which is an extension of Hellus and Schenzel's result (see \cite[Lemma 4.3]{pet1}).
\begin{prop}
With the notation in Theorem $\ref{4}$ suppose in addition that $K(R)$ exists and consider the following condition:
\begin{itemize}
\item[(f)] $H^i_I(K(R))= 0$ for all $i\neq c$.
\end{itemize}
then $(f)$ is equivalent to all the conditions of Theorem $\ref{4}$.
\end{prop}
\begin{proof}
If $R$ is a Cohen-Macaulay ring then $K(R)$ is a maximal Cohen-Macaulay module of finite injective dimension. Then the equivalence follows from Theorem \ref{4} by view of Proposition \ref{2.7}.
\end{proof}
Now the following Proposition is a generalization of an application of Hellus and Schenzel (see \cite[Lemma 4.3]{pet1}).
\begin{prop}
Let $(R,\mathfrak m)$ be a ring, $I\subseteq R$ be an ideal and $M\neq 0$ be an $R$-module such that $\grade(I, M)= c$. Suppose that $x\in I$ is $M$-regular then the following are equivalent:
\begin{itemize}
\item[(a)] $H^i_{IR/xR}(M/xM)= 0$ for all $i\neq c-1$.
\item[(b)] $H^i_I(M)= 0$ for all $i\neq c$ and $x$ is $\Hom_R(H^c_I(M), $E$)$- regular.
\end{itemize}
\end{prop}
\begin{proof}
Since $x$ is $M$-regular so we have the following short exact sequence

$$0\rightarrow M\mathop\rightarrow\limits^x M\rightarrow M/xM\rightarrow 0$$

Applying $\Gamma_I$ to this sequence yields the following long exact sequence of local cohomologies
\begin{equation}
...\rightarrow H^{i-1}_I(M/xM)\rightarrow H^i_I(M)\mathop\rightarrow\limits^x H^i_I(M)\rightarrow
H^i_I(M/xM)\rightarrow ...
\end{equation}
Now we prove that $(a)$ implies $(b)$.

If $i< c$ then there is the following exact sequence
\begin{equation}
0\rightarrow H^i_I(M)\mathop\rightarrow\limits^x H^i_I(M)
\end{equation}
Let $r\in H^i_I(M)$ then there exists $n\in \mathbb N$ such that $rI^n= 0$ it implies that $r= 0$ because of
sequence 5.2. It follows that $H^i_I(M)= 0$ for all $i< c.$ Similarly we can prove that $H^i_I(M)= 0$ for all $i> c$ so $H^i_I(M)= 0$ for all $i\neq c$. Now applying $\Hom_R(-,E)$ to sequence 5.1 and substituting $i= c$ then $x$ is $\Hom_R(H^c_I(M),E)$- regular follows from the following short exact sequence
\[
0\rightarrow \Hom_R(H^c_I(M),E)\mathop\rightarrow\limits^x \Hom_R(H^c_I(M),E) \rightarrow \Hom_R(H^{c-1}_{IR/xR}(M/xM),E)\rightarrow 0
\]
Now we will prove that the assertion $(b)$ implies $(a)$.

If $i\leq {c-2}$ or $i> c$ then it follows from sequence 5.1 and Independence of Base Theorem (see \cite{h}) that $H^i_{IR/xR}(M/xM)$= 0 for all $i\leq {c-2}$ or $i> c $. If $i= c$ then there is the following exact sequence
\[
H^c_I(M)\mathop\rightarrow^x H^c_I(M)\rightarrow H^c_{IR/xR}(M/xM)\rightarrow 0
\]
it induces the following exact sequence
\[
0\rightarrow \Hom_R(H^c_{IR/xR}(M/xM),E)\rightarrow \Hom_R(H^c_I(M),E) \mathop\rightarrow\limits^x \Hom_R(H^c_I(M),E)
\]
since $x$ is $\Hom_R(H^c_I(M),E)$- regular so $\Hom_R(H^c_{IR/xR}(M/xM),E)$ being the kernel of the morphism $\Hom_R(H^c_I(M),E)\mathop\rightarrow\limits^x \Hom_R(H^c_I(M),E)$ is zero. Therefore result follows from \cite[Remark 3.11]{h}.
\end{proof}
\begin{cor}
Let $(R,\mathfrak m)$ be a ring, $I\subseteq R$ be an ideal of $\grade(I, M)= c$ for an $R$-module $M\neq 0$. Suppose that $\underline{x}=
(x_1,...,x_j)\in I$ is an $M$-regular sequence for $1\leq j\leq c$ then the following are equivalent:
\begin{itemize}
\item[(a)] $H^i_{IR/\underline{x}}(M/\underline{x}M)= 0$ for all $i\neq c-j$.
\item[(b)] $H^i_I(M)= 0$ for all $i\neq c$ and $\underline{x}$ is $\Hom_R(H^c_I(M), $E$)$- regular.
\end{itemize}
\end{cor}


\end{document}